\documentclass[leqno,12pt]{amsart}
\usepackage[left=2.1cm,right=2.1cm,top=2.9cm,bottom=2.9cm]{geometry}
\usepackage{booktabs}
\usepackage[T1]{fontenc}
\usepackage[utf8]{inputenc}
\usepackage{lmodern}
\usepackage[spanish,english]{babel} 
\usepackage{amsmath, amssymb, amsfonts, amsthm,mathtools}
\usepackage{mathrsfs}
\usepackage{enumitem}
\usepackage{cite}
\usepackage{graphicx}
\usepackage{pdflscape}
\usepackage[table]{xcolor}
\usepackage[colorlinks=true]{hyperref}
\hypersetup{
  urlcolor=blue,
  linkcolor=blue,
  citecolor=red,
  anchorcolor=blue}

\newtheorem{theorem}{Theorem}[section]

\newtheorem{proposition}[theorem]{Proposition}

\theoremstyle{definition}
\newtheorem{definition}[theorem]{Definition}

\theoremstyle{remark}
\newtheorem{remark}{Remark}
\numberwithin{equation}{section}

\begin{document}
\title[Smoothing effects for fractional fast diffusion equations]{Fractional very fast diffusion equations in Lebesgue spaces: uniqueness and smoothing effects}
\author[M.-E.-M. Boudaoud, A. de Pablo \& F. Quirós]{Mohammed-El-Mahdi Boudaoud, Arturo de Pablo and Fernando Quir\'os}
\address{M.-E.-M. Boudaoud.
         Departamento  de Matem\'{a}ticas, Universidad Aut\'{o}noma de Madrid, 28049-Madrid, Spain.} 
\email{mohammed-el-mahdi.boudaoud@uam.es}

\address{A. De Pablo.
        Departamento de Matem\'{a}ticas, Universidad Carlos III de Madrid, 28911-Legan\'{e}s, Spain \& Instituto de Ciencias Matem\'aticas ICMAT (CSIC-UAM-UC3M-UCM), 28049-Madrid, Spain.}
\email{arturop@math.uc3m.es}

\address{F. Quir\'os.
        Departamento de Matem\'{a}ticas, Universidad Aut\'onoma de Madrid,
        \& Instituto de Ciencias Matem\'{a}ticas ICMAT (CSIC-UAM-UCM-UC3M),
        28049-Madrid, Spain.}
\email{fernando.quiros@uam.es}
\urladdr{https://matematicas.uam.es/~fernando.quiros}

\date{}

% ===============================

\begin{abstract}
We investigate forward and backward smoothing effects in Lebesgue spaces $L^p$ and $\mathcal{M}^p:=L^{p,\infty}$ for the Cauchy problem associated to the nonlinear and nonlocal fractional diffusion equation $\partial_t u+(-\Delta)^{\frac\sigma2}|u|^{m-1}u=0$ in $\mathbb{R}^N$, $0<\sigma<2$, in the very fast range $0<m\le m_c:=\frac{N-\sigma}{N}$.

We prove that very weak solutions have an $\mathcal{M}^p$--$L^\infty$ smoothing effect if $p>p^*:=\frac{N}{\sigma}(1-m)$, and we construct counterexamples showing the failure of any $L^p$--$\mathcal{M}^q$ forward ($q>p$) smoothing effect if $1< p< p^*$, $m<m_c$ or $L^1$--$\mathcal{M}^q$ ($q>1$) if $m=m_c$.

We also prove a backward $\mathcal{M}^p$--$L^1$ smoothing effect whenever $1\le p<p^*$, $m<m_c$, and we construct counterexamples showing that there is no $L^p$--$\mathcal{M}^q$ backward ($q<p$) smoothing effect if $p> p^*$.

Regarding the threshold value $p=p^*$, we prove that all solutions starting in $\mathcal{M}^{p^*}$ become extinct in finite time, and show the failure of any $\mathcal{M}^{p^*}$--$\mathcal{M}^q$ smoothing before extinction for any $q\neq p^*$.

The construction of counterexamples is based on new uniqueness and comparison results for very weak solutions, combined with the existence of self-similar solutions with suitable properties. The same approach yields new counterexamples for both forward and backward smoothing effects also in the local case $\sigma=2$.
\end{abstract}


\maketitle

%%%%%%%%%%%%%%%%%%%%%%%%%%%%%%%%%%%%%%%%%%%%%%%%%%%%%%%%%%%%%%%%%%%%%%%%%%%%%%%%%%%%%%%%%%%%%%%%%%%
\section{Introduction}

We give sharp conditions for the existence of 
%$L^p$-$L^q$ 
smoothing effects
%, $1\le p,q\le\infty$, 
for the nonlinear fractional Cauchy problem
\begin{equation}\label{eq:main}
    \begin{cases}
        \partial_t u+(-\Delta)^{\frac\sigma2}|u|^{m-1}u=0,&x\in \mathbb{R}^N,\ t>0,\\
        u(x,0)=u_0(x),&x\in\mathbb{R}^N,
    \end{cases}
\end{equation}
in the so-called \emph{very fast} range, $m\in (0,m_c]$, with  $m_c=\frac{N-\sigma}{N}>0$. Observe that necessarily $N>\sigma$. We assume a fractional order $\sigma\in(0,2)$ for the fractional Laplacian $(-\Delta)^{\frac\sigma2}$. 

We say that there is an $\mathbb{X}$--$\mathbb{Y}$ smoothing effect if \emph{all} solutions with initial datum in~$\mathbb{X}$  belong to $\mathbb{X}\cap\mathbb{Y}$ for almost every positive time. We will focus here initially on $L^p$--$L^q$ smoothing effects, $p\in[1,\infty)$, $p<q\le\infty$. However, it turns out that the spaces that mark the threshold between different behaviours are the weak Lebesgue spaces $L^{p,\infty}(\mathbb{R}^N)$, $p>1$, slightly larger than $L^p(\mathbb{R}^N)$, also known as Marcinkiewicz spaces, denoted here by $\mathcal{M}^p(\mathbb{R}^N)$. Therefore, we will also perform our analysis within them.  In order to simplify the reading and without loss of generality we restrict to nonnegative solutions.

The name very fast diffusion in the title comes from applications in the local case $\sigma=2$, i.e., $\partial_t u+(-\Delta)u^m=0$, which corresponds to slow diffusion if $m>1$, fast diffusion if $m\in (m_c,1)$, and very fast diffusion if $m\in(0,m_c)$. In this local case the critical exponent is $m_c:=\frac{(N-2)_+}{N}$. The first proof of a smoothing effect ($L^1$--$L^\infty$ if $m>m_c$) for local diffusion was given in \cite{veron:se,benilan:se}.  Later, in \cite{smoothbook}, Vázquez proved an $\mathcal{M}^p$--$L^\infty$ smoothing effect if $m<m_c$ under the condition $p>\frac{N}{2}(1-m)$. 

As for the nonlocal problem~\eqref{eq:main}, existence and uniqueness of a weak solution was proved in~\cite{b1}  both if $m>m_c$ and $u_0\in L^1(\mathbb{R}^N)$, or if $m\in (0,m_c]$ and  $u_0\in L^1(\mathbb{R}^N)\cap L^p(\mathbb{R}^N)$, $p>p^*:=\frac{N}{\sigma}(1-m)$. A crucial point of the existence proof is the fact that solutions  immediately enter into $L^\infty(\mathbb{R}^N)$, while staying in the same space as the initial datum: if $m\in(0,m_c]$ and $p>p^*$ there is an ($L^1\cap L^p$)--$L^\infty$ smoothing effect. 

The exponent $p^*$ arises naturally from scaling: problem~\eqref{eq:main} is invariant under the mapping $(x,t,u)\mapsto(\lambda^\beta x,\lambda t,\lambda^\alpha u)$ provided $\alpha(m-1)+\beta\sigma=1$, and for the datum $|x|^{-\frac{N}{p}}$ to be preserved one needs additionally $\alpha=\frac{N\beta}{p}$; this system 
has a unique solution $(\alpha,\beta)$ if and only if $p\neq p^*$. When $p\neq p^*$, this gives rise to self-similar solutions of the form $t^{-\alpha}S(t^{-\beta}|x|)$, which will play a key role in our analysis. The critical case $p=p^*$, where this scaling argument breaks down, corresponds instead to solutions in separated 
variables that become extinct in finite time; see Theorem~\ref{propo:selfsimilar}. Observe also that $p^*=1$ in the borderline case $m=m_c$.

\noindent\textsc{Forward smoothing effects. } Two natural questions arise. 
\begin{quote} 
    \centering
    \emph{Is there an $\mathcal{M}^p$--$L^\infty$ smoothing effect if $p>p^*$, without requiring $u_0\in L^1(\mathbb{R}^N)$?} 
\end{quote}
Notice that a positive answer would imply an $L^p$--$L^\infty$ smoothing effect, since $L^p(\mathbb{R}^N)\subset \mathcal{M}^p(\mathbb{R}^N)$.

Under this weaker assumption on the initial datum, the existence of a solution, in some suitable sense, has to be justified. The existence of very weak solutions with \emph{large} initial data, for all $m>0$, was addressed in \cite{b3}:  if the initial datum belongs to the weighted space $L^1_v(\mathbb{R}^N)=L^1(\mathbb{R}^N;v\,dx)$ with $v$ in a suitable class, there is a very weak solution that moreover stays in $L^1_v(\mathbb{R})$ for any later time. Uniqueness for such solutions when $m>m_c$ was proved in~\cite{jorge}. Uniqueness for $m\in(0,m_c]$ remained an open problem, and will be proved here; see Theorem~\ref{uniquenessthm}. Beyond its intrinsic interest, it has several consequences that will be essential in our analysis of smoothing effects. 

If the initial datum belongs to $L^p(\mathbb{R}^N)$ or $\mathcal{M}^p(\mathbb{R}^N)$, it belongs to $L^1_v(\mathbb{R}^N)$ for some weights $v$ in the admissible class. Hence, there is a very weak solution to the problem. Moreover, we will prove that this solution stays in the same space as the initial datum, $L^p(\mathbb{R}^N)$ or $\mathcal{M}^p(\mathbb{R}^N)$, for later times, thus showing the existence of a very weak $L^p$ or $\mathcal{M}^p$ solution in the sense of Definition~\ref{def:Lp.Mp}; see Theorems~\ref{solution in Lp} and~\ref{solution in Mp}. For the rest of the paper we will work in these classes of solutions. 

A partial answer to our first question was given in~\cite{vazquez:Mp_SE}, where the authors prove the desired $\mathcal{M}^p$--$L^\infty$ smoothing effect, but only for \emph{minimal} very weak solutions in $L^1_v(\mathbb{R}^N)$ (remember that uniqueness was not available). Thus, the possibility of larger solutions staying unbounded was still there. Our uniqueness result rules out this possibility, since there is no other solution different from the minimal one, and we obtain a positive answer to our question; see Theorem~\ref{thm:Forward Mp smoothing effect}.

We turn now our attention to the second question.
\begin{quote}
    \centering
    \emph{Is there any $L^p$ (or $\mathcal{M}^p$)--$L^q$ (or $\mathcal{M}^q$) smoothing effect with $q>p$, if $p\le p^*$?}  
\end{quote}
The answer is no. This is the content of Theorem~\ref{failure of $L^p$--$L^q$}(i). If $p<p^*$  we construct counterexamples in self-similar form. The same approach can also be used to give counterexamples in the local case, that were not available up to now for $q\in(p,p^*]$. The proof when $p= p^*$ follows an idea from \cite{chasvaz} for the local case, based on a comparison principle. Such comparison in our nonlocal setting was not known in the range $m\in(0,m_c]$; it will follow as a corollary of our uniqueness result; see Theorem~\ref{thm:comparison}.

\noindent\textsc{Backward smoothing effects. } Up to now we have only considered the possibility of passing from an $L^p$ (or $\mathcal{M}^p$) space to an $L^q$ (or $\mathcal{M}^q)$ space with $q>p$. These are the so-called \emph{forward} smoothing effects. But we could think about the possibility of entering into integrability spaces with $q<p$ (notice that this question made no sense for solutions living in $L^1(\mathbb{R}^N)\cap L^p(\mathbb{R}^N)$ as the ones considered in~\cite{b1}). Such a phenomenon is named as a \emph{backward} smoothing effect. New questions arise.

\begin{quote}
    \centering
    \emph{Is there any $L^p$ (or $\mathcal{M}^p$)--$L^q$ (or $\mathcal{M}^q$), backward ($q<p$) smoothing effect?}  
\end{quote}

Though there is no forward smoothing effect for initial data in $\mathcal{M}^{p}$ when $p< p^*$, we will show in Theorem~\ref{backthm} that there is a backward $\mathcal{M}^{p}$--$L^1$ smoothing effect: $\mathcal{M}^p$ solutions will become instantaneously integrable. The corresponding analysis for the local case was developed in~\cite{smoothbook}.  One may then wonder whether such backward smoothing effects could happen if $p\ge p^*$.  At this respect we use the self-similar solutions constructed in~\cite{vazquez:Mp_SE} to show that this is not the case; see Theorem~\ref{failure of $L^p$--$L^q$}(ii). A similar approach can be used to give counterexamples for the local case, that were not available up to now.

The following table summarizes our results on the forward and backward 
\(L^p\)–\(\mathcal{M}^q\) and \(\mathcal{M}^p\)–\(L^q\)  smoothing effects. 
\renewcommand{\arraystretch}{2}
\[
\begin{array}{c|c|c|c|c}
 & p<q<\infty & q=\infty & p>q>1 & q=1\\
\hline
1\le p<p^{*}
& \cellcolor{red!10} L^{p}-\mathcal{M}^{q}\ \times
& \cellcolor{red!10} L^{p}-L^{\infty}\ \times
& \cellcolor{green!10} \mathcal{M}^{p}-L^{q}\ \checkmark
& \cellcolor{green!10}\mathcal{M}^{p}-L^{1}\ \checkmark
\\
\hline
p=p^{*}
& \cellcolor{red!10} \mathcal{M}^{p^*}-\mathcal{M}^{q}\ \times
& \cellcolor{red!10} \mathcal{M}^{p^*}-L^{\infty}\ \times
& \cellcolor{red!10} \mathcal{M}^{p^*}-\mathcal{M}^{q}\ \times
& \cellcolor{red!10} \mathcal{M}^{p^*}-L^{1}\ \times
\\
\hline
p>p^{*}
& \cellcolor{green!10} \mathcal{M}^{p}-L^{q}\ \checkmark
& \cellcolor{green!10}\mathcal{M}^{p}-L^{\infty}\ \checkmark
& \cellcolor{red!10} L^{p}-\mathcal{M}^{q}\ \times
& \cellcolor{red!10} L^{p}-L^{1}\ \times
\\
\hline
\end{array}
\]

\smallskip

\centerline{\small $\checkmark$: smoothing effect holds;\quad $\times$: smoothing effect fails.}

\medskip 

\noindent\textsc{The critical exponent $p^*$. } The results above lead to the following question:

\begin{quote}
    \centering
    \emph{What happens for $p=p^*$?}  
\end{quote}

For this critical value we have neither forward nor backward standard smoothing effects. However, we will prove in Theorem~\ref{extinction:Lp and Mp} that $\mathcal{M}^{p^*}$ is an \emph{extinction class}: \emph{all} solutions with initial data in $\mathcal{M}^{p^*}$ become identically zero in finite time. This finite-time extinction phenomenon may be regarded as a delayed smoothing effect.  

If $p\neq p^*$, there are examples of $L^{p}$ solutions that do not become extinct; see Theorem~\ref{extinction:Lp and Mp}. Thus, $\mathcal{M}^{p^*}$ is the only extinction class in our framework. These issues for the local case were analyzed in~\cite{smoothbook}.

\noindent\textsc{The local case. }  When $\sigma=2$, a counterexample for the forward smoothing effect was only known if $m=m_c$ and $p=p^*=1$; see~\cite{chasvaz}. The same construction of Theorem~\ref{failure of $L^p$--$L^q$}(i)-(ii) can be used to construct counterexamples to both the forward and backward smoothing effects in the local case for $m\in(0,m_c)$.

\subsection*{Organization of the paper.} After giving some preliminary definitions and results in Section~\ref{sect:Preliminaries}, we devote Section~\ref{sec:uniqueness,comparison} to the proof of one of our main results, uniqueness of very weak weighted solutions, as well as two comparison results (standard comparison and concentration comparison) that will be essential for the rest of the paper. We move then to  Section~\ref{sec:theory in Lebesgues sp}, where we prove that problem~\eqref{eq:main} is well posed in Lebesgue spaces by showing existence and uniqueness of $L^p$ and $\mathcal{M}^p$ solutions, and obtain, as a corollary of the results of the previous section in combination with previous work from~\cite{vazquez:Mp_SE}, extinction and forward smoothing effects. Section~\ref{sec:positives results} is devoted to the proof of backward smoothing effects. Counterexamples proving the lack of smoothing effects outside the good ranges are constructed in Section~\ref{sec:failure of se}. We end with an appendix in which we prove the continuity of $L^p$ solutions for positive times when $p>p^*$.

%%%%%%%%%%%%%%%%%%%%%%%%%
\section{Preliminaries}
\label{sect:Preliminaries}

In this section we recall the concepts of fractional Laplacian, Marcinkiewicz  spaces, concentration comparison and the different definitions of solution: very weak, $L^p$ and $\mathcal{M}^p$ very weak, and weighted very weak solutions. 

\noindent\emph{Notations. } Throughout, we will frequently view the solution $u$ of \eqref{eq:main} as a map from time into a Lebesgue space. In this setting, $u(t)$ will denote, for almost every $t>0$,  the spatial function given by $u(t)(x)=u(x,t)$  for almost every $x\in\mathbb{R}^N$. Letter $c$ will denote a constant  that may change from line to line. By $g\sim h$ we mean that there are constants $c_1,\,c_2>0$ such that $c_1 h\le g\le c_2 h$. $B_R=B_R(0)$ denotes the ball of radius $R$.

%%%%%%%%%%%%
\subsection*{Fractional Laplacian.} 
The nonlocal operator $(-\Delta)^{\frac\sigma2}$, known as the fractional Laplacian of order $\sigma\in(0,2)$, is well defined in the Schwartz space via the Fourier transform, see \cite{Stein-1970},
\[
\mathcal{F}\big( (-\Delta)^{\frac\sigma2} h \big)(\xi)
= |\xi|^{\sigma}\,\mathcal{F}(h)(\xi).
\]
It also admits the following integral representation for bounded smooth functions
$$
    (-\Delta)^{\frac\sigma2}h(x)=C_{N,\sigma}\text{P.V.}\int_{\mathbb{R}^N}\frac{h(x)-h(y)}{|x-y|^{N+\sigma}}\,dy,
$$
where $C_{N,\sigma}$ is a normalization constant, and P.V. stands for principal value. 

A weak definition of the fractional Laplacian can be given for functions $h$ belonging to $\dot{H}^{\frac\sigma2}(\mathbb{R}^N)$, the homogeneous Sobolev space equipped with the seminorm 
\begin{equation*}
[h]_{\dot{H}^{\frac\sigma2}}=\left(\frac{C_{N,\sigma}}{2}\int_{\mathbb{R}^N}\int_{\mathbb{R}^N}\frac{(h(x)-h(y))^2}{|x-y|^{N+\sigma}}\,dydx\right)^{1/2}.
\end{equation*}
A fundamental inequality associated with this operator is the Hardy-Littlewood-Sobolev estimate (see~\cite{Lieb}):
\[
    \|h\|_{\frac{Np}{N-p\sigma}}\leq c\|(-\Delta)^{\frac\sigma2}h\|_p, \quad h\in\dot{H}^{\frac\sigma2}(\mathbb{R}^N).
\]

Rather than working strictly within the Schwartz class or the Sobolev space, we consider the weighted space 
$\mathcal{L}_{\sigma}=L^1_\rho(\mathbb{R}^N)$, with $\rho=(1+|x|^2)^{-\frac{(N+\sigma)}2}$, 
%$\rho=v_{N+\sigma}$, where the family of weights  is $v_\alpha(x)=(1+|x|^2)^{-\frac{\alpha}2}$, 
where, given a weight $v$, we define
$$
L^1_v(\mathbb{R}^N):=\left\{h:\mathbb{R}^N\rightarrow\mathbb{R} \;\text{such that}\;\int_{\mathbb{R}^N}|h(x)|v(x)\,dx<\infty\right\}.
$$ 
Although the fractional Laplacian of a smooth compactly supported function does not have compact support, its size is controlled by $\rho$. Therefore the operator $(-\Delta)^{\frac\sigma2}$ is well defined in that weighted space in the distributional sense, 
$$
\int_{\mathbb{R}^N}(-\Delta)^{\frac\sigma2}hg=\int_{\mathbb{R}^N}h(-\Delta)^{\frac\sigma2}g,\qquad h\in \mathcal{L}_{\sigma},\quad g\in C^\infty_{\textup{c}}(\mathbb{R}^N).
$$
The operator $(-\Delta)^{-\frac\sigma2}$, known as the Riesz potential, is defined for $\psi\in L^p(\mathbb{R}^N)$ by
\begin{equation}\label{riesz}
    (-\Delta)^{-\frac\sigma2}\psi(x)=C_{N,-\sigma}\int_{\mathbb{R}^N}\frac{\psi(y)}{|x-y|^{N-\sigma}}\,dy,
\end{equation}
if $0<\sigma<\frac Np$. For $\frac Np<\sigma<N$ the potential is understood in the distributional sense.

For more information on this topic, see, for instance, \cite{DiNezza-Palatucci-Valdinoci-2012,Landkof-1972,Stein-1970}.

%%%%%%%%
\subsection*{Rearrangements and mass concentration.}

Let $f$ be a real measurable function on $\mathbb{R}^N$. The \emph{distribution function} $\mu_f$ of $f$ is defined by
\[ 
    \mu_f(k)=|\{x\in\mathbb{R}^N:\ |f(x)|>k\}|, \qquad k\ge 0,
\]
where the right-hand side denotes set measure, and the \emph{decreasing rearrangement} of $f$ by
\[
    f^*(s)=\sup\{k\ge 0:\ \mu_f(k)>s\}, \qquad s>0.
\]
From this definition it turns out that
$\mu_{f^*}=\mu_f$ ($f$ and $f^*$ are equi-distributed) and $f^*$ is exactly
the generalized inverse of $\mu_f$. Moreover,
$$
\|f\|_p^p=\|f^*\|_p^p=p\int_0^\infty\lambda^{p-1}\mu_f(\lambda)\,d\lambda,
$$
for every $1\le p<\infty$. Hardy-Littlewood's inequality is also useful,
\begin{equation}\label{H-L}
\int_{\mathbb{R}^N}|f(x)g(x)|\,dx\le \int_0^\infty f^*(s)g^*(s)\,ds.
\end{equation}

Furthermore, if $\omega_N$ is the measure of the unit ball in $\mathbb{R}^N$, we define the function
\[
 f^\#(x)=f^*(\omega_N |x|^N),\quad x\in \mathbb{R}^N,
\]
that will be called the \emph{spherical decreasing rearrangement} of $f$.
We then say that $f$ is rearranged if $f=f^\#$.

For the comparison argument needed later in Section~\ref{sec:positives results}, the concept of mass concentration is fundamental. Specifically, we say that a function $h$ is less concentrated than $g$, denoted $ h \prec g$, if both $h$ and $g$ are radially symmetric, belong to $L^1_{\mathrm{loc}}(\mathbb{R}^N)$, and satisfy, for all $R>0$
\[
    \int_{B_R}h(x)\,dx\leq \int_{B_R}g(x)\,dx.
\]
If both $h$ and $g$ are rearranged, it can be easily checked that  $h \prec g$ implies $\|h\|_p\le \|g\|_p$ for every $1\le p\le\infty$. 

The above concepts can be defined analogously in domains $\Omega\subset\mathbb{R}^N$. The same can be done in the next paragraph. A good reference for these topics is \cite{vazquez-sym}.
%%%%%%%%
\subsection*{Marcinkiewicz spaces.}
For $p\in(1,\infty)$, the space $\mathcal{M}^p(\mathbb{R}^N)$, (also known as the weak $L^p$ space, denoted as $L^{p,\infty}(\mathbb{R}^N)$) consists of all functions $f\in L^1_{\textup{loc}} (\mathbb{R}^N)$ such that
\[\int_K |f(x)|\,dx\leq C |K|^{\frac{p-1}{p}},\]
for some constant $C$ and all measurable subsets $K\subset \mathbb{R}^N$ with finite measure. The smallest such constant $C$ defines a quasi-norm; see for instance \cite{Grafakos-2014}. An equivalent quasi-norm that will be helpful for us is
\begin{equation}\label{weaknorm}
\|f\|_{\mathcal{M}^p}=\sup_{\lambda>0}\lambda(\mu_f(\lambda))^{\frac1p}=\sup_{s>0}s^{\frac1p}f^*(s).
\end{equation}
It is easy to check the strict inclusion $
L^p(\mathbb{R}^N)\varsubsetneq \mathcal{M}^p(\mathbb{R}^N),
$
for every $1< p<\infty$. The typical example is the pure power $U_0(x)=|x|^{-\frac{N}{p}}$, which satisfies $U_0\in
\mathcal{M}^p(\mathbb{R}^N)\setminus L^p(\mathbb{R}^N)$. This is the most concentrated function in $\mathcal{M}^p$, useful in comparison arguments.

%%%%%%%%%%%%%%
\subsection*{Very weak solutions}

In the case of the classical Laplacian, the notion of very weak solution only requires
$u_0,\, u(t)\in L^1_{\mathrm{loc}}(\mathbb{R}^N)$; see \cite{herreropierre}. However, when dealing with problem~\eqref{eq:main}, the nonlocal character of $(-\Delta)^{\frac\sigma2}$ makes it necessary to introduce the weighted space $\mathcal{L}_{\sigma}$. 
\begin{definition}\label{def1}
A \emph{very weak solution} to problem~\eqref{eq:main} with initial datum $u_0\in L^1_{\textup{loc}}(\mathbb{R}^N)$ is a function $u$ such that:

\noindent\textup{(i)} $u\in C([0,\infty);L^1_{\textup{loc}}({\mathbb{R}^N}))$ and $u^m\in L^1_{\textup{loc}}((0,\infty);\mathcal{L}_{\sigma})$;

\noindent\textup{(ii)} for all $\phi\in C^{\infty}_{\textup{c}}(\mathbb{R}^N\times (0,\infty))$ we have the identity
\begin{equation}\label{vwe}
    \int_{0}^{\infty}\int_{\mathbb{R}^N}u\partial_t\phi \,dxdt-\int_{0}^{\infty}\int_{\mathbb{R}^N}u^m(-\Delta)^{\frac\sigma2}\phi \,dxdt=0;
\end{equation}

\noindent\textup{(iii)} $u(t)\to u_0$ in $L^1_{\text{loc}}(\mathbb{R}^N)$.
\end{definition}

\subsection*{Very weak \texorpdfstring{$L^p$ and $\mathcal{M}^p$}{Lp and Mp} solutions} 
We are interested in solutions that live in $L^p$ ($1\le p\le\infty$) or $\mathcal{M}^p$ ($1< p<\infty$) for almost every $t>0$. This suggests the following definition.
\begin{definition}\label{def:Lp.Mp}
    \textup{(i)} A \emph{very weak $L^p$ solution} (or simply $L^p$ solution) to problem~\eqref{eq:main} with initial datum $u_0\in L^p(\mathbb{R}^N)$  is a very weak solution $u$ to the problem with that initial datum such that $u\in L^1_{\rm loc}((0,\infty);L^p (\mathbb{R}^N))$.

\noindent\textup{(ii)} A \emph{very weak $\mathcal{M}^p$ solution} (or simply $\mathcal{M}^p$ solution) to problem~\eqref{eq:main} with initial datum $u_0\in \mathcal{M}^p(\mathbb{R}^N)$   is a very weak solution $u$ to the problem with that initial datum such that $u\in L^1_{\rm loc}((0,\infty);\mathcal{M}^p (\mathbb{R}^N))$.
\end{definition}
We will prove existence and uniqueness within these classes. As a tool to this aim we will consider \emph{weighted very weak solutions}, a concept already used in~\cite{b3} in the context of problem~\eqref{eq:main} that we recall next.

\subsection*{Weighted very weak solutions}

The existence of very weak solutions is not guaranteed for initial data that are just in  $L^1_{\rm loc}$.
However, it is available under further integrability assumptions, namely $u_0\in L^1_v(\mathbb{R}^N)$ with $v$ in the class of admissible weights
\[
 \Theta:=\left\{v(x)=v_\alpha(x)=(1+|x|^2)^{-\frac\alpha2}\ \text{for some}\ \alpha\in\left(N-\tfrac{\sigma}{1-m},N+\tfrac{\sigma}{m}\right)\right\}.
\]
Indeed, as proved in~\cite{b3}, under this assumption on the initial datum there is a weak solution \emph{that stays} in $L^1_v(\mathbb{R}^N)$ for later times. This motivates the following definition.
\begin{definition}\label{vww:definition}
    A \emph{weighted very weak solution} (\emph{weighted solution} in the sequel) to problem~\eqref{eq:main}, with weight $v\in\Theta$, is a very weak solution $u$ to the problem  with initial datum $u_0 \in L^1_v(\mathbb{R}^N)$  that satisfies further
    $u\in L^1_{\textup{loc}}((0,\infty);L^1_{v}({\mathbb{R}^N}))$. 
\end{definition}

Weighted solutions satisfy an important contraction property. If $u,\hat{u}$ are two weighted solutions with the same weight $v\in \Theta$ (without loss of generality we can change the weight of one of them, just taking the minimum of the exponents $\alpha$), then, for all $m\in(0,1)$,
\begin{equation}\label{weight+}
    \left(\int_{\mathbb{R}^N}(u-\hat{u})_+(x,t)v(x/R)\right)^{1-m}\leq \left(\int_{\mathbb{R}^N}(u-\hat{u})_+(x,\tau)v(x/R)\right)^{1-m}+\frac{C(t-\tau)}{R^{\sigma-N(1-m)}},
\end{equation}
for $0\leq\tau\leq t<\infty$; see~\cite[Theorem 3.5]{jorge}.
If $m>m_c$ inequality~\eqref{weight+} gives comparison and uniqueness just by letting $R$ go to infinity. The extension of such uniqueness and comparison results to the range $m\in(0,m_c]$ is one of the main features of this paper. 

On the other hand, the (strict) inclusion $\mathcal{M}^p(\mathbb{R}^N)\subset L^1_v(\mathbb{R}^N)$ holds
for some weights $v$ in the admissible class, depending on $p$. Clearly any constant belongs to $L^1_v(\mathbb{R}^N)\setminus\mathcal{M}^p(\mathbb{R}^N)$.  The formal argument of the inclusion is as follows: if $f\in \mathcal{M}^p(\mathbb{R}^N)$, then by~\eqref{weaknorm} its decreasing rearrangement satisfies $f^*(s)\le \|f\|_{\mathcal{M}^p}s^{-\frac1p}$. Take a weight $v=v_\alpha$. Then, analogously $v^*(s)\le c s^{-\frac\alpha N}$. In fact $v^*(s)=(1+(s/\omega_N)^{\frac2N})^{-\frac\alpha2}\sim  s^{-\frac\alpha N}$ at infinity. Now, applying Hardy-Littlewood inequality~\eqref{H-L},
$$
\int_{\mathbb{R}^N}|f(x)v(x)|\,dx\le \int_0^\infty f^*(s)v^*(s)\,ds\le c_1+c_2\int_1^\infty s^{-\frac1p}s^{-\frac\alpha N}\,ds<\infty,
$$
provided $\alpha>\frac{N(p-1)}p$. Thus, if $p<p^*$ any decay in the class $\Theta$ is allowed, while for $p\ge p^*$ we must restrict to the interval $\alpha\in(N-\frac{N}{p},N+\frac \sigma m)$. 

In summary, if the initial datum belongs to $L^p$ or $\mathcal{M}^p$, it also belongs to a weighted space~$L^1_v(\mathbb{R}^N)$ for certain weights. Hence, it gives rise to a weighted solution. We will show that this solution remains in the same space, $L^p$ or $\mathcal{M}^p$, for later times. Hence,  it is an $L^p$, respectively $\mathcal{M}^p$, solution. In the case of $L^p$ solutions, the $L^p$ norm does not increase with time.

%%%%%%%%%%%%%%

%%%%%%%%%%%%%%%%%%%%%%%%%%%%%%%%%%%%%%%%%%%%%%%%%
\section{\texorpdfstring
{Uniqueness and comparison for weighted very weak solutions in the very singular regime $m\le m_c$.}{Uniqueness and comparison for weighted very weak solutions in the very singular regime m <= mc.}}\label{sec:uniqueness,comparison}

In this section we begin by establishing uniqueness of weighted (very weak) solutions, following the approach of Herrero and Pierre \cite{herreropierre} in the local case. We then prove a comparison result using an approximation argument. Finally, we extend the concentration comparison result of \cite{volzone} to  Marcinkiewicz spaces. 

Our proof of uniqueness relies on the following technical result. 

\begin{proposition}\label{subharmonic}
 Let $u$ and $\hat{u}$ be two weighted solutions of the problem~\eqref{eq:main} with the same initial datum. Then the function 
\begin{equation}\label{w}
    w(x,t)=\int_{0}^{t}|u^{m}-\hat{u}^{m}|(x,s)\,ds
\end{equation} 
is $\frac\sigma2$--subharmonic in the very weak sense.
\end{proposition}

\begin{proof}
We begin by proving that $(-\Delta)^{\frac\sigma2}w\le0$ when $u$ and $\hat{u}$ are bounded smooth solutions. For general solutions the computation is only formal, and will be justified in the very weak sense later.

We multiply the expression resulting from the subtraction of the equations satisfied by $u$ and~$\hat u$ by $\operatorname{sign}(u-\hat{u})$ to obtain
$$
 \operatorname{sign}(u-\hat{u})\partial_t (u-\hat{u})+\operatorname{sign}(u-\hat{u})(-\Delta)^{\frac\sigma2}(u^{m}-\hat{u}^{m})=0.
$$
Using Kato's inequality (see \cite{cordobaforuniqueness}), we deduce that
\begin{equation}\label{eq:Kato}
    \partial_t |u-\hat{u}|\leq-(-\Delta)^{\frac\sigma2}|u^{m}-\hat{u}^{m}|.
\end{equation}
Integrating in time we obtain    
$$
        |u-\hat{u}|(x,t)\leq -\int_{0}^{t}(-\Delta)^{\frac\sigma2}|u^{m}-\hat{u}^{m}|(x,s)\,ds=-(-\Delta)^{\frac\sigma2}w(x,t).
$$
Thus, $w$ is subharmonic. 

For solutions satisfying the equation in a very weak sense we must show that 
\[
    \int_{\mathbb{R}^N}w(-\Delta)^{\frac\sigma2}\psi\leq 0\quad\textup{for all } 0\leq\psi\in C^{\infty}_{\textup{c}}(\mathbb{R}^N).
\]
By Definition~\ref{def1}, for all $\phi\in C^{\infty}_c(\mathbb{R}^N\times (0,\infty))$ we have
\begin{equation}\label{weak_diff}
    \int_{0}^{\infty}\int_{\mathbb{R}^N}(u-\hat{u})\partial_t\phi \,dxdt-\int_{0}^{\infty}\int_{\mathbb{R}^N}(u^m-\hat{u}^m)(-\Delta)^{\frac\sigma2}\phi \,dxdt=0.
\end{equation}
We will choose the test function $\phi$ in terms of $\psi$. Let $v=v_\alpha\in \Theta$ be the common weight for the two weighted solutions $u,\hat{u}$. Let $b>0$ such that $\phi\equiv 0$ for $t>b$. Observe that once we fix $\phi$ we clearly have $|\partial_t\phi|\le Kv$ in $\mathbb{R}^N\times[0,b]$ for some constant $K>0$. We approximate $u,\hat{u}$ by smooth functions $u_\varepsilon,\hat{u}_\varepsilon$ satisfying, 
\begin{equation}\label{5}
 \int_{0}^{b}\int_{\mathbb{R}^N}|\hat{u}-\hat{u}_{\varepsilon}|v\,dxdt<\varepsilon,\quad \int_{0}^{b}\int_{\mathbb{R}^N}|u-u_{\varepsilon}|v\,dxdt<\varepsilon.
\end{equation}
We begin by decomposing the first integral in~\eqref{weak_diff},
\begin{align*}
\displaystyle \int_{0}^{b}\int_{\mathbb{R}^N}(u-\hat{u})\partial_t\phi&\displaystyle= \int_{0}^{b}\int_{\mathbb{R}^N}(u-u_{\varepsilon})\partial_t\phi+ \int_{0}^{b}\int_{\mathbb{R}^N}(u_{\varepsilon}-\hat{u}_{\varepsilon})\partial_t\phi+ \int_{0}^{b}\int_{\mathbb{R}^N}(\hat{u}_{\varepsilon}-\hat{u})\partial_t\phi \\[3mm]
&\displaystyle\leq \int_{0}^{b}\int_{\mathbb{R}^N}|u-u_{\varepsilon}||\partial_t\phi|+ \int_{0}^{b}\int_{\mathbb{R}^N}|\hat{u}-\hat{u}_{\varepsilon}||\partial_t\phi|+ \int_{0}^{b}\int_{\mathbb{R}^N}(u_{\varepsilon}-\hat{u}_{\varepsilon})\partial_t\phi.
\end{align*}
Thus, 
\begin{equation}\label{6}
    \int_{0}^{b}\int_{\mathbb{R}^N}(u-\hat{u})\partial_t\phi\leq 2K\varepsilon+ \int_{0}^{b}\int_{\mathbb{R}^N}(u_{\varepsilon}-\hat{u}_{\varepsilon})\partial_t\phi.
    \end{equation}
Similarly, for the nonlocal term 
\begin{align*}
    \int_{0}^{b}\int_{\mathbb{R}^N}(u^m-\hat{u}^m)(-\Delta)^{\frac\sigma2}\phi&= \int_{0}^{b}\int_{\mathbb{R}^N}(u^m-u_{\varepsilon}^m)(-\Delta)^{\frac\sigma2}\phi\\
    &\quad+ \int_{0}^{b}\int_{\mathbb{R}^N}(u_{\varepsilon}^m-\hat{u}_{\varepsilon}^m)(-\Delta)^{\frac\sigma2}\phi+ \int_{0}^{b}\int_{\mathbb{R}^N}(\hat{u}_{\varepsilon}^m-\hat{u}^m)(-\Delta)^{\frac\sigma2}\phi.
\end{align*}
Therefore,
\begin{equation}\label{eq:decomposition}
\begin{aligned}
    \int_{0}^{b}\int_{\mathbb{R}^N}(u_{\varepsilon}^m-\hat{u}_{\varepsilon}^m)(-\Delta)^{\frac\sigma2}\phi&\leq
    \int_{0}^{b}\int_{\mathbb{R}^N}|(u^m-u_{\varepsilon}^m)(-\Delta)^{\frac\sigma2}\phi|
    +\int_{0}^{b}\int_{\mathbb{R}^N}|(\hat{u}_{\varepsilon}^m-\hat{u}^m)(-\Delta)^{\frac\sigma2}\phi|\\
    &\quad +\int_{0}^{b}\int_{\mathbb{R}^N}(u^m-\hat{u}^m)(-\Delta)^{\frac\sigma2}\phi.
\end{aligned}
\end{equation} 
 Using the concavity of the power $m<1$ and then Hölder's inequality, we estimate
\begin{align*}
    \int_{0}^{b}\int_{\mathbb{R}^N}|(u^m-u_{\varepsilon}^m)(-\Delta)^{\frac\sigma2}\phi|&\le c \int_{0}^{b}\int_{\mathbb{R}^N}\frac{v^m}{v^m}|u-u_{\varepsilon}|^m|(-\Delta)^{\frac\sigma2}\phi|
 \\
 &\le c
 \left( \int_{0}^{b}\int_{\mathbb{R}^N}|u-u_{\varepsilon}|v\right)^m
 \left( \int_{0}^{b}\int_{\mathbb{R}^N}\frac{|(- \Delta)^{\frac\sigma2}\phi|^{\frac{1}{1-m}}}{v^{\frac{m}{1-m}}}\right)^{1-m}.
\end{align*}
For every $\phi$ smooth and compactly supported function,  $(-\Delta)^{\frac\sigma2}\phi$ has a minimal decay at infinity, $|x|^{-N-\sigma}$. Then, we split the last integral as
\begin{equation}\label{eq:bound.phi.v}
    \begin{aligned}
        \int_{\mathbb{R}^N} \frac{|(-\Delta)^{\frac\sigma2}\phi|^{\frac{1}{1-m}}}{v^{\frac{m}{1-m}}}&=
        \int_{|x|< 1} \frac{|(-\Delta)^{\frac\sigma2}\phi|^{\frac{1}{1-m}}}{v^{\frac{m}{1-m}}}+ \int_{|x|> 1} \frac{|(-\Delta)^{\frac\sigma2}\phi|^{\frac{1}{1-m}}}{v^{\frac{m}{1-m}}}\\&\leq
        c_1+ \int_{|x|>1} \frac{c}{|x|^{\frac{N+\sigma-m\alpha}{1-m}}}<c.
    \end{aligned}
\end{equation}
Observe that $\alpha<N+\frac{\sigma}{m}$ implies $\frac{N+\sigma-m\alpha}{1-m}>N$. Now, by~\eqref{5},
\[
    \int_{0}^{b}\int_{\mathbb{R}^N}|(u^m-u_{\varepsilon}^m)(-\Delta)^{\frac\sigma2}\phi|\le c\varepsilon^m.
\]
We have the same estimate for the term involving $\hat u$ and $\hat u_\varepsilon$. Then~\eqref{eq:decomposition} becomes
$$
    \quad \int_{0}^{b}\int_{\mathbb{R}^N}(u_{\varepsilon}^m-\hat{u}_{\varepsilon}^m)(-\Delta)^{\frac\sigma2}\phi\leq 2c\varepsilon^m+ \int_{0}^{b}\int_{\mathbb{R}^N}(u^m-\hat{u}^m)(-\Delta)^{\frac\sigma2}\phi.
$$
Using~\eqref{6} and the weak formulation~\eqref{vwe}, this yields
\begin{equation}\label{7}
    \int_{0}^{b}\int_{\mathbb{R}^N}(u_{\varepsilon}^m-\hat{u}_{\varepsilon}^m)(-\Delta)^{\frac\sigma2}\phi\leq\varepsilon+2c\varepsilon^m+ \int_{0}^{b}\int_{\mathbb{R}^N}(u_{\varepsilon}-\hat{u}_{\varepsilon})\partial_t\phi.
\end{equation}
Put  $\phi=\Phi'_h(u_{\varepsilon}-\hat{u}_{\varepsilon})\tilde{\phi}$, where $\Phi_h$ is a smooth approximation of the absolute value function and $0\leq \tilde{\phi}\in C^\infty(\mathbb{R}^N\times(0,b))$ will be chosen in terms of $\psi$. Then,
integrating by parts in time twice,
$$
    \int_{0}^{b}\int_{\mathbb{R}^N}(u_{\varepsilon}-\hat{u}_{\varepsilon})\partial_t\phi= - \int_{0}^{b}\int_{\mathbb{R}^N}(\Phi'_h(u_{\varepsilon}-\hat{u}_{\varepsilon})\tilde{\phi})\partial_t(u_{\varepsilon}-\hat{u}_{\varepsilon})= \int_{0}^{b}\int_{\mathbb{R}^N}\Phi_h(u_{\varepsilon}-  \hat{u}_{\varepsilon})\partial_t\tilde{\phi}.
$$
Letting $h\to0^+$, we get
$$
    \displaystyle\int_{0}^{b}\int_{\mathbb{R}^N}(u_{\varepsilon}-\hat{u}_{\varepsilon})\partial_t\phi =  \int_{0}^{b}\int_{\mathbb{R}^N}|u_{\varepsilon}-  \hat{u}_{\varepsilon}|\partial_t\tilde{\phi}.
$$
As to the left-hand side of~\eqref{7}, integrating by parts in space, letting $h\to0^+$ and using Kato's inequality, we get
$$
    \int_{0}^{b}\int_{\mathbb{R}^N}(u_{\varepsilon}^m-\hat{u}_{\varepsilon}^m)(-\Delta)^{\frac\sigma2}\left(\Phi'_h(u_{\varepsilon}-\hat{u}_{\varepsilon})\tilde{\phi}\right)\ge\int_{0}^{b}\int_{\mathbb{R}^N}\tilde{\phi}(-\Delta)^{\frac\sigma2}|u_{\varepsilon}^m-\hat{u}_{\varepsilon}^m|.
$$
We then perform another integration by parts, and \eqref{7} becomes
\begin{equation}\label{eq:aproximadas}
    \int_{0}^{b}\int_{\mathbb{R}^N}|u_{\varepsilon}^m-\hat{u}_{\varepsilon}^m|(-\Delta)^{\frac\sigma2}\tilde{\phi}\leq \varepsilon+2c\varepsilon^m+ \int_{0}^{b}\int_{\mathbb{R}^N}|u_{\varepsilon}-\hat{u}_{\varepsilon}|\partial_t\tilde{\phi}.
\end{equation}
On the other hand, 
\begin{align*}
    \Big|\int_{0}^{b}\int_{\mathbb{R}^N}(|u^m&-\hat{u}^m|-|u_{\varepsilon}^m-\hat{u}_{\varepsilon}^m|)(-\Delta)^{\frac\sigma2}\tilde{\phi}\Big| \le  \int_{0}^{b}\int_{\mathbb{R}^N}(|u^m-u_{\varepsilon}^m|+|\hat{u}^m-\hat{u}_{\varepsilon}^m|)|(-\Delta)^{\frac\sigma2}\tilde\phi|\\
    &\overset{\text{Concavity}}{\le}
    \int_{0}^{b}\int_{\mathbb{R}^N}(|u-u_{\varepsilon}|^m+|\hat{u}-\hat{u}_{\varepsilon}|^m)\frac{v^m}{v^m}|(-\Delta)^{\frac\sigma2}\tilde\phi| \\
    &\hskip6pt\overset{\text{Hölder}}{\le}\int_{0}^{b}\left(\int_{\mathbb{R}^N}\frac{|(-\Delta)^{\frac\sigma2}\tilde\phi|^{\frac{1}{1-m}}}{v^{\frac{m}{1-m}}}\right)^{1-m}
    \left(\left(\int_{\mathbb{R}^N}|u-u_{\varepsilon}|v\right)^m+\left(\int_{\mathbb{R}^N}|\hat{u}-\hat{u}_{\varepsilon}|v\right)^m\right)\\
    &\hskip-2pt\overset{\text{\eqref{5}}+\text{\eqref{eq:bound.phi.v}}}{\le} c\varepsilon^m.
\end{align*}
Thus, letting $\varepsilon\to 0^+$ in~\eqref{eq:aproximadas},
\begin{equation}\label{8}
    \int_{0}^{b}\int_{\mathbb{R}^N}|u^m-\hat{u}^m|(-\Delta)^{\frac\sigma2}\tilde{\phi}\leq \int_{0}^{b}\int_{\mathbb{R}^N}|u-\hat{u}|\partial_t\tilde{\phi}.
\end{equation}

Fix any positive times $0<\tau<t<b$. We take $\tilde{\phi}=F_{k,\tau}(t)\psi(x),$ where $F_{k,\tau}$ is a smooth approximation of the characteristic function $\chi_{[\tau,t]}(t)$, and $0\leq\psi\in C^{\infty}_0(\mathbb{R}^N)$. Letting $k\to\infty$ in~\eqref{8} we get
$$
    \int_{\tau}^{t}\int_{\mathbb{R}^N}|u^m-\hat{u}^m|(-\Delta)^{\frac\sigma2}\psi\leq \int_{\mathbb{R}^N}\psi|u-\hat{u}|(\tau) -\int_{\mathbb{R}^N}\psi|u-\hat{u}|(t).
$$
Since $u$ and $\hat{u}$ take the same initial value in $L^1_{\textup{loc}}$, we can pass to the limit as $\tau\to 0$ making the first term on the right vanish, so we obtain,
$$
 \int_{\mathbb{R}^N}w(-\Delta)^{\frac\sigma2}\psi\leq 0\quad\textup{for all }0\leq\psi\in C^{\infty}_c(\mathbb{R}^N),
$$
that is, $w$ is $\frac{\sigma}{2}$--subharmonic in the very weak sense.
\end{proof}

This result will be combined with an estimate for $\frac\sigma2$-subharmonic functions due to Silvestre~\cite[Proposition 2.19]{silvestre}, stated below. The estimate is written in terms of a bounded smooth function~$\gamma$ constructed in that paper, that satisfies $\gamma\ge0$ and $\gamma(x)\sim |x|^{-N-\sigma}$ for $|x|$ large. Put
$$
 \gamma_R (x) =R^{-N}\gamma(x/R),
$$
where $\gamma$ is Silvestre's function.

\begin{proposition}[Silvestre~\cite{silvestre}]\label{proposilvestre}
    A function $f\in \mathcal{L}_{\sigma}$ satisfies $(-\Delta)^{\frac\sigma2}f\leq0$ in an open set $\Omega$ if and only if it
    is upper semicontinuous in $\Omega$  and $f(x_0)\leq f\ast \gamma_{R}(x_0)$ for any $x_0$ in $\Omega$ and $R\leq \operatorname{dist}(x_0, \partial\Omega)$.
\end{proposition}

With this we are now able to establish the uniqueness within the class of weighted solutions for any $m\in (0,1)$. The result is new in the range~$m\in(0,m_c]$.

\begin{theorem}[Uniqueness]\label{uniquenessthm}
   Let $u,\hat{u}$ be two weighted solutions to the equation~\eqref{eq:main} with the same initial datum. Then $u\equiv\hat{u}$.
\end{theorem}

\begin{proof}
We know from Definition~\ref{def1} that $u^m, \hat{u}^m \in L^1_{\mathrm{loc}}((0,\infty); \mathcal{L}_{\sigma}(\mathbb{R}^N))$. This implies that the function $w$ defined in~\eqref{w} satisfies $w(t) \in \mathcal{L}_{\sigma}$ for almost every $t>0$. Moreover, by Propositions~\ref{subharmonic} and~\ref{proposilvestre}, we obtain that 
\[
    w(x,t)\leq \int_{\mathbb{R}^N}^{}w(y,t)\gamma_R(x-y)\, dy.
\]
Using again the concavity of the function $s\mapsto s^m$, we estimate
$$
    w(x,t)\le c \int_{0}^{t}\int_{\mathbb{R}^N}^{}|u-\hat{u}|^m(y,s)\gamma_R(x-y)\frac{v^m(y/R)}{v^m(y/R)}\,dyds,
$$
where $v=v_\alpha\in\Theta$ is a common weight for $u$ and $\hat{u}$.
Applying Hölder’s inequality and using estimate~\eqref{weight+} with absolute value and $\tau=0$ (the first term on the right-hand side vanishes, since $u$ and $\hat u$ share the same initial datum), we obtain 
$$
  w(x,t)\le \frac{ct^{\frac{1}{1-m}}}{R^{N-mN+{\frac{m\sigma}{1-m}}}}\left(\int_{\mathbb{R}^N}^{}\frac{\gamma^{\frac{1}{1-m}}((x-y)/R)}{v^{\frac{m}{1-m}}(y/R)}\,dy\right)^{1-m}
  =\frac{t^{\frac{1}{1-m}}}{R^{{\frac{m\sigma}{1-m}}}}\left(\int_{\mathbb{R}^N}^{}\frac{\gamma^{\frac{1}{1-m}}(x/R-z)}{v^{\frac{m}{1-m}}(z)}\,dz\right)^{1-m}.
$$
Observe that the last integral is bounded independently of $R$ for $|z|\le 2|x|$ since $v$ is locally strictly positive and $\gamma$ is bounded. So, we only have to consider $|z|\ge 2|x|$; in this region we have 
$$
    v^{\frac{m}{1-m}}(z)\sim |z|^{-\frac{m\alpha}{1-m}}\quad\text{and}\quad \gamma^{\frac{1}{1-m}}(x/R-z)\le c|z|^{-\frac{N+\sigma}{1-m}},
$$
so we need $\frac{N+\sigma-m\alpha}{1-m}>N$, which is fulfilled, since $\alpha<N+\frac{\sigma}{m}$. Hence, 
$$
 w(x,t)\le ct^{\frac{1}{1-m}}R^{-\frac{m\sigma}{1-m}}\rightarrow0 \quad\text{when $R$ tends to infinity.}
$$
Therefore, $w(x,t)\equiv 0$, and we conclude $u\equiv \hat {u}$. 
\end{proof}

As a consequence of uniqueness we can prove two comparison principles.

\begin{theorem}[Comparison]\label{thm:comparison}
    Let $u,\hat{u}$ be two very weak weighted solutions to equation~\eqref{eq:main}, with the same weight $v\in\Theta$, such that $u(x,0)\leq \hat{u}(x,0)$ in $\mathbb{R}^N$. Then
    $$
        u(x,t)\leq \hat{u}(x,t)\quad \text{in }\mathbb{R}^N,\;t>0.
    $$
\end{theorem}

\begin{proof}
 Let $u_{0,n}, z_{0,n} \in L^1(\mathbb{R}^N)\cap L^{\infty}(\mathbb{R}^N)$ be nondecreasing sequences converging monotonically and in $L^1_v(\mathbb{R}^N)$ respectively to $u_0$ and $z_0$, and such that $u_{0,n}\le z_{0,n}$. Let $u_n,z_n$ denote the solutions corresponding respectively to $u_{0,n},z_{0,n}$. By the comparison result in~\cite[Theorem 2.4]{b1}, $u_n(x,t)\leq z_n(x,t)$. Letting $n\to\infty$ we get the result, since, by Theorem~\ref{uniquenessthm}, the (monotone) limits $u=\lim_{n\to\infty}u_n$ and $z=\lim_{n\to\infty}z_n$ are, respectively, the unique solutions to~\eqref{eq:main} with initial data $u_0$ and $z_0$.
\end{proof}

\begin{remark}
    This argument also applies to the classical case involving the standard Laplacian. The theory of very weak solutions developed by Herrero and Pierre in \cite{herreropierre} provides both existence and uniqueness in this setting. These solutions are locally integrable in space, and a comparison principle for very weak solutions can be established via approximation through smooth solutions.
\end{remark}

 We also prove comparison by concentration, an adaptation of Theorem~5.3 in \cite{volzone}. It plays a key role when dealing with data in $\mathcal{M}^p$.

\begin{proposition}[Concentration comparison]\label{propo:concentrationcomparison}
Let $u$ be the solution to problem~\eqref{eq:main} with initial datum $u_0\in \mathcal{M}^p(\mathbb{R}^N)$, $u_0\ge 0$, and let $z$ be the solution to the same problem with $u_0$ replaced by~$u_0^{\#}$. If $z(t)\in L^p(\mathbb{R}^N)$, respectively $z(t)\in \mathcal{M}^p(\mathbb{R}^N)$, for some $p\in[1,\infty]$ and all $t>0$, then
\begin{equation}\label{eq:concentration.ordering}
    u^{\#}(t)\prec z(t)\quad\textup{for all }t>0,
\end{equation}
and $u(t)\in L^p(\mathbb{R}^N)$, respectively $u(t)\in \mathcal{M}^p(\mathbb{R}^N)$, with
\[
    \|u(t)\|_p\le \|z(t)\|_p,\quad \textup{respectively } \|u(t)\|_{\mathcal{M}^p} \le \|z(t)\|_{\mathcal{M}^p},\quad \text{for all }t>0.
\]

\end{proposition}

\begin{proof} We first notice that $\|u_0\|_{\mathcal{M}^p}=\|u^\#_0\|_{\mathcal{M}^p}$ and $\|u_0\|_{p}=\|u^\#_0\|_{p}$. 
Since $\mathcal{M}^p(\mathbb{R}^N)\subset L^1_v(\mathbb{R}^N)$ we can take a nondecreasing sequence $u_{0,n}$ in
$L^{1}(\mathbb{R}^N)\cap L^{\infty}(\mathbb{R}^N)$ approaching $u_0$ in $L^1_v(\mathbb{R}^N)$. Defining now $z_{0,n}=u_{0,n}^{\#}$, it is easy to check that $z_{0,n}\to z_0$ in a monotone way in $L^1_v(\mathbb{R}^N)$. For any $t>0$ the function \(u^{\#}_n(t)\) is radially nonincreasing by construction. Also, from~\cite{volzone} we have that \(z_n(t)\) is radially nonincreasing, and  $u_n^{\#}(t)\prec z_n(t)$ for any $t>0$. By the monotone convergence theorem we have
$$
z_n(t)\uparrow \bar{z}(t),\qquad
u_n(t)\uparrow \bar{u}(t)
\quad\text{in }L^1_v(\mathbb{R}^N).
$$
By uniqueness,  $\overline{z}=z$, $\overline{u}=u$; in particular, \(z(t)\) and \(u(t)\) are radially nonincreasing. Moreover, convergence holds almost everywhere in \(\mathbb{R}^N\). Furthermore, convergence is strong in \(L^1_{\mathrm{loc}}(\mathbb{R}^N)\). Hence, for every \(R>0\),
\[
u_n^{\#}\prec z_n
\quad\Longrightarrow\quad
\int_{B_R}u_n^{\#}(|x|,t)\,dx
\le
\int_{B_R}z_n(|x|,t)\,dx.
\]
Passing to the limit as \(n\to\infty\), we obtain
\[
\int_{B_R}u^{\#}(|x|,t)\,dx
\le
\int_{B_R}z(|x|,t)\,dx;
\]
that is, $u^{\#}(t)\prec z(t)$, and the result follows. 
\end{proof}

%%%%%%%%%%%%%%%%%%%%%%%%%
\section{Theory in Lebesgue and Marcinkiewicz Spaces}\label{sec:theory in Lebesgues sp}

In this section we establish the existence and uniqueness of $L^p$ ($p\ge1$) and $\mathcal{M}^p$  ($p>1$) solutions, and present some of their properties. We first devote our attention to $L^p$ solutions.

\begin{theorem}\label{solution in Lp}
    Let $u_0\in L^p(\mathbb{R}^N)$, $p\in[1,\infty)$. There is a unique $L^p$ solution to~\eqref{eq:main} with initial datum $u_0$.  Moreover, $\|u(t)\|_p$ is essentially nonincreasing in time. 
\end{theorem}

\begin{proof}
Since $L^p(\mathbb{R}^N)\subset L^1_v(\mathbb{R}^N)$ for some weight $v\in \Theta$, the existence and uniqueness of a weighted solution follow immediately from the previous theory. 
Furthermore, using an approximation argument with weak solutions together with Proposition 8.5 in~\cite{b1}, we deduce $\|u(t_2)\|_p\le \|u(t_1)\|_p$ for almost every pair $0\le t_1<t_2$. This shows that the quantity $\|u(t)\|_p$ is essentially nonincreasing in time and the weighted solution is in fact an $L^p$ very weak solution.
\end{proof}

We consider now the broader class  $\mathcal{M}^p$ of functions. In this setting concentration comparison plays a central role. We thus study in detail the special solution $U$ corresponding to the worst case in $\mathcal{M}^p$, the problem with initial datum the most concentrated function, a multiple of the power $|x|^{-\frac{N}{p}}$. As we have said, it will be used in the proof of the smoothing effects by comparison.

\begin{theorem}\label{propo:selfsimilar}
Let $U$ be the solution to problem~\eqref{eq:main} with initial datum $U_0(x) = A|x|^{-\frac{N}{p}}$, $p\in(1,\infty)$, $A>0$. If $p\neq p^*$ then $U$ is  self-similar,
\begin{equation}\label{sol-selfsimilar}
U(x,t)=t^{-\alpha}S(t^{-\beta}|x|),
\qquad
\alpha=\frac{N}{p\sigma-(1-m)N},
\quad
\beta=\frac{p\alpha}{N},
\end{equation}
while if $p=p^*$ the solution is explicit,
\begin{equation}\label{sol-extinction}
    U(x,t)=c(T-t)^{\frac{1}{1-m}}|x|^{-\frac{\sigma }{1-m}},\qquad cT^{\frac{1}{1-m}}=A,\quad c=c(m,\sigma,N).
\end{equation}
Moreover, if $ p < p^*$, then $
U(x,t) \sim|x|^{-\frac{N}{p}}$ at the origin, while if  $ p > p^*$, then $U(x,t)$ is bounded in $x$ for every $t>0$ and the behaviour $
U(x,t) \sim|x|^{-\frac{N}{p}}$ occurs at infinity. If $p\neq p^*$ the solution is positive for all times, and if $p=p^*$ the solution vanishes identically in the finite time $t=T$.
\end{theorem}

\begin{proof}
In~\cite{vazquez:Mp_SE}, the authors construct the minimal solution to~\eqref{eq:main} with the desired initial data when $p> p^*$, and show that it has the above self-similar form. However, uniqueness was not yet available at that time. Our uniqueness result closes this gap.

As for the subcritical range $p\in (1,p^*)$, the existence of a minimal solution follows by a standard approximation procedure. Since the problem has some kind of homogeneity, self-similarity will follow from uniqueness (this idea works also for the supercritical range). Just observe that for the above exponents $\alpha,\,\beta$, and for every $\lambda>0$, the function $V_\lambda(x,t)=\lambda^\alpha U(\lambda^\beta x,\lambda t)$ is also a solution with initial value $V_\lambda(x,0)=U_0(x)$. Thus, uniqueness yields $U(x,t)=V_\lambda(x,t)$. Putting $\lambda=1/t$ we obtain the desired self-similar expression with $S(\rho)=U(\xi,1)$, $\rho=|\xi|$.

If $ p < p^*$, we have $ \alpha, \beta< 0 $, and thus, for any $|x|\ne0$,
$$
    A|x|^{-\frac{N}{p}}=\lim_ {t\to0}U(x,t)=\lim_ {t\to0}t^{-\alpha}S(t^{-\beta}|x|)=|x|^{-\frac{N}{p}}\lim_{\rho\to0}\rho^{\frac{N}{p}}S(\rho).
$$
Therefore $S(\rho)\sim \rho^{-\frac Np}$ at the origin, which gives the behaviour of $U$ at any time $t>0$. 

The argument for $p>p^*$ is analogous, using the fact that $ \alpha, \beta > 0 $. The boundedness follows from the elliptic equation satisfied by the profile; see~\cite{vazquez:Mp_SE} for the details. 

In the case $p=p^*$ the denominator in the exponents $\alpha,\beta$ vanishes, so instead of the form~\eqref{sol-selfsimilar} we look for solutions in separated variables form, and we find expression~\eqref{sol-extinction} using the formula
\[
    (-\Delta)^{\frac\sigma2}|x|^{- \frac{m\sigma}{1-m}}=\kappa_{m,\sigma,N}|x|^{-\frac{m\sigma}{1-m}-\sigma};
\]
see for instance \cite{lieb-loss}. The coefficient $c$ is then given by
\[
    c^{1-m}(1-m)=\kappa_{m,\sigma,N}.\qedhere
\]
\end{proof}

These solutions were constructed in the local case in \cite{smoothbook}, using phase-plane techniques. A precise behaviour is obtained also at infinity in that paper when $p<\frac{N(1-m)}2$. In our fractional case we obtain next some estimate at infinity in the subcritical range $p<p^*=\frac{N(1-m)}\sigma$; though it is not expected to be optimal, it is enough for our purposes.

\begin{theorem}\label{thm global bound}
    Let $1<p<p^*$ and let $U$ be the solution to problem~\eqref{eq:main} with initial datum $U_0(x)=A|x|^{-\frac Np}$. Then, there exists $c>0$ such that
    \begin{equation}\label{U:estimate}
        U(x,t)\le c\,t^{-\frac 1 m}|x|^{-\frac{N-p\sigma}{pm}},\quad t>0.
    \end{equation}
\end{theorem}

\begin{proof}
Integrating~\eqref{eq:Kato} with $u=U$ and $\hat u\equiv0$ in $[0,t]$ we get 
\[
    U(t)\le U_0-\int_0^t (-\Delta)^{\frac\sigma2}U^m(s)\,ds,
\]
whence
\begin{equation}\label{eq:relation.w.u_0}
    (-\Delta)^{\frac\sigma2}w\leq U_0,\quad \text{where }w(x,t)=\int_{0}^{t}U^{m}(x,s)\,ds.
\end{equation}
Let $g=(-\Delta)^{-\frac\sigma2}U_0$, the Riesz potential of $U_0$, as defined in~\eqref{riesz}, which is finite, since $p<\frac{N}{\sigma}$. In fact,
see \cite{lieb-loss},
\begin{equation}\label{rieszradial}
    g=A(-\Delta)^{-\frac{\sigma}{2}}\left(|x|^{-\frac{N}{p}}\right)=AC_{\frac{N}{p},\sigma}|x|^{-\frac{N}{p}+\sigma}\in \mathcal{L}_{\frac\sigma2}.
\end{equation} 
We rewrite~\eqref{eq:relation.w.u_0} as  $(-\Delta)^{\frac\sigma2}(w-g)\leq 0$. This inequality is only known to hold in the distributional sense. Hence, in order to deduce from it that $w\le g$ we have to work a bit more. Recall that $U\in \mathcal{M}^p(\mathbb{R}^N)\subset L_v^1(\mathbb{R}^N)$, for every weight $v\in\Theta$. Take $v=v_\alpha$ with $\alpha<N$. 
Using that $(w-g)\in \mathcal{L}_{\sigma}(\mathbb{R}^N)$, we may use Silvestre's Proposition~\ref{proposilvestre}, to get
\begin{align*}
    w(x,t) &\leq g(x) + \int_{\mathbb{R}^N} \big( w(y,t) - g(y)\big)\,\gamma_R(x-y)\,dy \\
    &\leq g(x) + \int_0^t \int_{\mathbb{R}^N} U^m(y,s)\gamma_R(x-y)\frac{v^m(y/R)}{v^m(y/R)}\,dy\,ds
    \\ &\leq g(x) + \frac{c}{R^N}\int_0^t 
    \left( \|U_0 v(\cdot/R)\|_{1}^{1-m} + \frac{s}{R^{\sigma-N(1-m)}} \right)^{\tfrac{m}{1-m}}
    \left( \int_{\mathbb{R}^N} \frac{\gamma^{\tfrac{1}{1-m}}((x-y)/R)}{v^{\tfrac{m}{1-m}}(y/R)}\,dy \right)^{1-m}\,ds\\
    &\leq g(x) + \frac{c}{R^{Nm}}\int_0^t 
    \left( \|U_0 v(\cdot/R)\|_{1}^{1-m} + \frac{s}{R^{\sigma-N(1-m)}} \right)^{\tfrac{m}{1-m}}\,ds.
\end{align*}
We have used inequality~\eqref{weight+} with $\tau=0$. The integral involving $\gamma$ and $v$ has been shown to be bounded in the proof of Theorem~\ref{uniquenessthm}. Now using the triangular inequality, $(a+b)^\alpha\leq \max\{1,2^{\alpha-1}\}(a^\alpha+b^\alpha)$ for all $a,b>0$ and $\alpha>0$, we have
\begin{align*}
    w(x,t)&\leq g(x) + \frac{c}{R^{Nm}}\int_0^t \left( \|U_0 v(\cdot/R)\|_{1}^{m} + \frac{s^{\tfrac{m}{1-m}}}{R^{\tfrac{\sigma m}{1-m}-mN}}\right) \, ds\\&
    \leq g(x)+\frac{c\,t}{R^{Nm}} \left(\left(\int_{B_1}|x|^{-\frac Np}\,dx\right)^m+\left(\int_{B_1^c}|x|^{-\frac Np}|x/R|^{-\alpha}\,dx\right)^m \right)+ \frac{c\,t^{\tfrac{1}{1-m}}}{R^{\tfrac{\sigma m}{1-m}}}\\&
    \leq g(x)+\frac{c\,t}{R^{Nm}}\left(1+R^{\alpha m}\right)+\frac{c\, t^{\tfrac{1}{1-m}}}{R^{\tfrac{\sigma m}{1-m}}}.
\end{align*}
Since $\alpha<N$, letting $R\to\infty$  we conclude that $w \leq g$. 

On the other hand, as consequence of the homogeneity of the parabolic operator, nonnegative solutions satisfy
$$
    \frac{U(t_2)}{U(t_1)}\leq \left(\frac{t_2}{t_1}\right)^{\frac{1}{1-m}},\quad 0<t_1<t_2; 
$$
see \cite{benilan-crandall}. Combining this with $w\leq g$ we get 
\begin{equation*}
    g(x)\ge \int_{0}^{t}U^m(x,s)\,ds\ge  \int_{0}^{t}U^m(x,t)\left(\frac{s}{t}\right)^{\frac{m}{1-m}}\,ds=(1-m)t\,U^m(x,t),
\end{equation*}
that together with~\eqref{rieszradial} gives
\begin{equation*}
    U(x,t)\le \left(\frac{g(x)}{(1-m)t}\right)^{\frac{1}{m}}=c \ t^{-\frac 1 m}|x|^{-\frac{N-p\sigma}{pm}}\qquad \text{for any $t>0$.}\qedhere
\end{equation*} 
\end{proof}

\begin{remark}
    The global estimate~\eqref{U:estimate} is true for $1<p<\frac N\sigma$. However, it  gives information only if $p<p^*$ and only for large $|x|$, see Theorem~\ref{propo:selfsimilar}.
\end{remark}

\begin{theorem}\label{U in Mp}
   let $U$ be the solution to problem~\eqref{eq:main} with initial datum $U_0(x)=A|x|^{-\frac Np}$, $p>1$. Then $U(t)\in\mathcal{M}^p(\mathbb{R}^N)$ for every $t>0$. 
\end{theorem}

\begin{proof}     
\noindent\underline{Case $p< p^*$.} \ 
We know from self-similarity, Theorem~\ref{propo:selfsimilar}, that \( U(t) \) behaves in this range near the origin like the initial value, $U(x,t)\sim|x|^{-\frac{N}{p}}$. Hence, for any fixed \(R>0\), we have \(U(t)\in \mathcal{M}^p(B_R)\). It remains to control the norm for large values of $|x|$. To this end, we estimate the \( \mathcal{M}^p(B_R^c) \)-norm using~\eqref{U:estimate}:
\begin{align*}
    \|U(t)\|_{\mathcal{M}^p(B^c_R)} 
    &\le c\,t^{-\frac{1}{m}}\sup_{\lambda>0} \lambda \left|\left\{|x|\ge R : |x|^{-\frac{N-p\sigma}{pm}} > \lambda \right\}\right|^{\frac{1}{p}} \\
    & \le c\,t^{-\frac{1}{m}}\sup_{\lambda<R^{-\frac{N-p\sigma}{pm}}} \lambda \left|\left\{R\le |x|<\lambda^{-\frac{pm}{N-p\sigma}} \right\}\right|^{\frac{1}{p}}\le c\,t^{-\frac{1}{m}}\sup_{\lambda<R^{-\frac{N-p\sigma}{pm}}}\lambda^{1-\frac{Nm}{N-p\sigma}}.
\end{align*}
The supremum is finite provided $\frac{Nm}{N-p\sigma}\le 1$, i.e., $p\le p^*$.

\noindent\underline{Case $p=p^*$.} \ In this critical case the solution is explicit $U(x,t)=c(t)|x|^{-\frac\sigma{1-m}}=c(t)|x|^{-\frac N{p^*}}$, see~\eqref{sol-extinction}, and satisfies $U(t) \in \mathcal{M}^{p^*}(\mathbb{R}^N)$.

\noindent\underline{Case $p>p^*$.} \
We first recall that $U(t)$ is bounded for any $t>0$. Also, according again to Theorem~\ref{propo:selfsimilar}, the asymptotic behavior of $U(x,t)$ for large $|x|$ is the same as the initial value,
$U(x,t) \sim |x|^{-\frac{N}{p}}$.
This implies that $U(t) \in \mathcal{M}^p(\mathbb{R}^N)$. 
\end{proof}

As a corollary we obtain existence of $\mathcal{M}^p$ solutions.
\begin{theorem}\label{solution in Mp}
    Given $u_0\in \mathcal{M}^p(\mathbb{R}^N)$, $p\in(1,\infty)$, there exists a unique $\mathcal{M}^p$ solution to problem~\eqref{eq:main}. 
\end{theorem}

\begin{proof}
     Existence of a unique very weak solution is already known; it only remains to prove that this  solution satisfies $u(t)\in \mathcal{M}^p(\mathbb{R}^N)$ for all $t>0$. This follows immediately by concentration comparison with the worst-case scenario $U_0(x) = A |x|^{-\frac{N}{p}}$, for some $A>0$ depending on $u_0$, and the previous result.
\end{proof}

The uniqueness result of the previous section, together with the $\mathcal{M}^p$--$L^\infty$ smoothing effect for \emph{minimal} weighted solutions proved in \cite{vazquez:Mp_SE} in the case $p>p^*$, yields the following theorem.

\begin{theorem}[Forward smoothing effect]\label{thm:Forward Mp smoothing effect}
   Let $u$ be a solution to~\eqref{eq:main} with initial datum $u_0\in 
   \mathcal{M}^p(\mathbb{R}^N)$, $p>p^*$. Then $u(t)\in L^\infty(\mathbb{R}^N)$ for any $t>0$. In particular, we have the smoothing effect $\mathcal{M}^p$--$L^q$ for every $p<q\le\infty$. 
\end{theorem}

Next we characterize the extinction spaces.

\begin{theorem}\label{extinction:Lp and Mp}
    \textup{(i)} $\mathcal{M}^{p^*}(\mathbb{R}^N)$ is an extinction space for problem~\eqref{eq:main}.
    
    \noindent \textup{(ii)} For any $p\in [1,\infty]$, $p\neq p^*$, there are $L^p$ solutions to~\eqref{eq:main} that never become extinct.
\end{theorem}

\begin{proof} (i) The explicit $\mathcal{M}^{p^*}$ solution $U$ with initial datum $U(0)=A|x|^{-\frac N{p^*}}$ belongs to $\mathcal{M}^{p^*}(\mathbb{R}^N)$ for any $t>0$ and vanishes in a finite time $T_A=cA^{1-m}$; see~\eqref{sol-extinction}. We can therefore deduce the result by concentration comparison, Proposition~\ref{propo:concentrationcomparison}: given $u_0\in\mathcal{M}^{p}(\mathbb{R}^N)$, take $A=\|u_0\|_{\mathcal{M}^p}$, and the solution vanishes in a time not bigger than $T_A$.

\noindent (ii) \underline{Case $p>p^*$.} If we pick an exponent $r\in(p^*,p)$ and the solution corresponding to the power $U_0(x)=|x|^{-\frac Nr}$, we have by Theorem~\ref{propo:selfsimilar}, that $U(t) \in L^{p}(\mathbb{R}^N)$. Then, by considering the time-shifted solution $u(t) = U(t+\tau)$ with $\tau > 0$, we have that $u(0) \in L^{p}(\mathbb{R}^N)$ and is positive for all times. It follows that $L^{p}(\mathbb{R}^N)$ is not an extinction space.

\noindent\underline{Case $p<p^*$.} The argument is similar, now using the backward  $\mathcal{M}^p-L^1$ smoothing effect for $p<p^*$ (Theorem~\ref{backthm}, proved in the next section), which guarantees that the time-shifted solution $u(t)=U(t+\tau)$ has initial datum in $L^p(\mathbb{R}^N)$.
\end{proof}

%%%%%%%%%%%%%%%%%%%%%%%%%%%%%%%%%%%%%%%%%%%%%%%%%%%%%%%%%%%%%%%%%%%%%%%%%
\section{Backward smoothing effects}\label{sec:positives results}

In this section we investigate the possibility of having backward smoothing effects, meaning that a solution with an initial datum in some $L^p$, enters instantaneously in some $L^q$ with $q\in[1,p)$. We  show that an $L^p$--$L^1$ backward smoothing effect holds when $p<p^*$ (then necessarily $p^*>1$, i.e., $m<m_c$). Notice that this implies an $L^p$--$L^q$ backward smoothing effect for any $q\in[1,p)$. In fact we can do better, proving the smoothing effect for initial data in the Marcinkiewicz space $\mathcal{M}^p$,~$p\in(1,p^*)$. The corresponding result in the local case was proved in~\cite[Theorem~5.15]{smoothbook}.  

\begin{theorem}[Backward smoothing effect]\label{backthm}
   Let $u_0\in\mathcal{M}^p(\mathbb{R}^N)$, $1<p<p^*$, then the corresponding $\mathcal{M}^p$ solution to~\eqref{eq:main} is integrable in space for any $t>0$. We then have the smoothing effect  $\mathcal{M}^p$--$L^q$ for every $1\le q<p$.
\end{theorem}

\begin{proof}
By concentration comparison (Proposition~\ref{propo:concentrationcomparison}), it suffices to
prove the result for the worst-case solution $U$ with initial datum $U_0(x)=A|x|^{-\frac{N}{p}}$. The proof depends on the sign of $p-p_m$, with $p_m:=N/(mN+\sigma)$.

\noindent\underline{Case $p \in (1, p_m)$.}  The exponent $\frac{N-p\sigma}{pm}$ in estimate~\eqref{U:estimate} exceeds $N$ if and only if $p<p_m$, so $U(t)\in L^1(B_1^c)$ if $p$ is in that range. Since $U(t)\in L^1_{\textup{loc}}(\mathbb{R}^N)$ always holds, we conclude $U(t)\in L^1(\mathbb{R}^N)$ when $p<p_m$.

\noindent\underline{Case $p =p_m$.} In this case~\eqref{U:estimate} becomes $U(x,t)\le
c\,t^{-\frac{1}{m}}|x|^{-N}$, whence $U(t)\in L^q(B_1^c)$ for every $q>1$. Combined with the behaviour at the origin, $U(x,t)\sim|x|^{-\frac{N}{p_m}}$ (Theorem~\ref{propo:selfsimilar}; recall that $p<p^*$), we get $U(t)\in\mathcal{M}^q(\mathbb{R}^N)$ for every $q\in(1,p_m)$, and we then conclude by the case $p<p_m$ already proved.

\noindent\underline{Case $p\in(p_m,p^*)$.} In this range we use an iterative argument. 
We want to define an increasing sequence
\[
    q_{k+1}=F(q_k), \qquad q_1=p_m,\quad \lim_{k\to\infty}q_k=p^*.
\]
We will show that if $U_0\in\mathcal{M}^p(\mathbb{R}^N)$ with $p\in(q_k,q_{k+1}]$  for some $k$ (which always holds if $p\in(p_m,p^*)$), then $U(t)\in\mathcal{M}^{q_k}(\mathbb{R}^N)$,
so that after $k$ iterations we obtain $U(t)\in\mathcal{M}^{q_1}(\mathbb{R}^N)
=\mathcal{M}^{p_m}(\mathbb{R}^N)$, and we conclude by the case $p=p_m$ already proved.

Let $p\in(q_k,q_{k+1}]$. Since $p>q_k$, then $\mathcal{M}^{p}(B_R)\subset\mathcal{M}^{q_k}(B_R)$ for all $R>0$. Hence, it suffices to control the $\mathcal{M}^{q_k}$-norm at infinity. Using~\eqref{U:estimate},
\begin{equation}\label{eq:q_k}
    \|U(t)\|_{\mathcal{M}^{q_k}(B_R^c)}
    \le c\,t^{-\frac{1}{m}}
    \sup_{\lambda>0}\lambda\,\bigl|\bigl\{|x|>R:|x|^{-\frac{N-p\sigma}{pm}}>\lambda\bigr\}\bigr|^{\frac{1}{q_k}}.
\end{equation}
The right-hand side is finite provided $q_k\, \frac{N-p\sigma}{pm}\ge N$, that is, if $p\le\frac{Nq_k}{mN+\sigma q_k}$, and in that case we conclude that $U(t)\in\mathcal{M}^{q_k}(\mathbb{R}^N)$. Thus, it is enough to define $q_k$ by the nonlinear recurrence $q_{k+1}=\frac{Nq_k}{mN+\sigma q_k}$. It can be made linear by the change $z_k=\frac 1{q_k}$, which is easily solved and gives the desired sequence,
\[
    q_k=\frac{N}{\mu m^{k}+\nu}, \qquad \nu=\frac{\sigma}{1-m},\quad \mu=N-\nu>0.
\]

The $\mathcal{M}^p$--$L^q$ effect for $q\in(1,p)$  follows by interpolation.
\end{proof}

%%%%%%%%%%%%%%%%%%%%%%%%%%%%%%%%%%%%%%%%%%%%%%%%%%%%%%%%%%%%%%%%%%%%%%%%%
\section{Failure of the smoothing effects}\label{sec:failure of se}

The purpose of this section is twofold: to construct counterexamples showing the failure of the $L^p$--$\mathcal{M }^q$ smoothing effect for problem~\eqref{eq:main} when $1\le p < p^*$ and $q>p$; and to construct counterexamples showing the failure of the $L^p$--$\mathcal{M }^q$ backward smoothing effect when $p > p^*>1$ and $q<p$. In those results  $m<m_c$. If $m=m_c$, for which $p^*=1$, we show the failure of the smoothing $L^1$--$\mathcal{M }^q$, $q>1$.

\begin{theorem}\label{failure of $L^p$--$L^q$}
    \noindent\textup{(i)} Let  $m\in(0,m_c)$ and $p\in[1, p^*)$. There is no $L^p$--$\mathcal{M }^q$ forward smoothing effect for any $q\in(p,\infty)$. Moreover, there is no $L^p$--$L^\infty$ forward smoothing effect.

    \noindent\textup{(ii)} Let  $m\in(0,m_c)$ and $p\in(p^*,\infty)$. There is no $L^p$--$\mathcal{M}^q$ backward smoothing effect for any $q\in(1,p)$. Moreover, there is no $L^p$--$L^1$ backward smoothing effect.

    \noindent\textup{(iii)} Let  $m=m_c$. There is no $L^1$--$\mathcal{M }^q$ smoothing effect for any $q>1$.
\end{theorem}

\begin{proof}

\noindent\textup{(i)} Let $U$ be the self-similar solution with initial datum the power $U_0(x)=|x|^{-\frac{N}{r}}$ with $r\in(p,\min\{p^*,q\})$. By Theorem~\ref{propo:selfsimilar},  $U(x,t)\sim|x|^{-\frac{N}{r}}$ at the origin for all $t>0$, whence $U(t)\not\in \mathcal{M}^q(\mathbb{R}^N)$ for any $t>0$; besides, $U(t)\not\in L^\infty(\mathbb{R}^N)$. On the other hand, by the backward smoothing effect, Theorem~\ref{backthm}, $U(\tau)\in L^p(\mathbb{R}^N)$ for all $\tau>0$. The function $u(t)=U(t+\tau)$ gives thus the desired counterexample.

\noindent\textup{(ii)}
The proof is analogous to previous the case, taking $r\in (\max\{p^*,q\},p)$, using  the behavior of the corresponding self-similar solution $U$ at infinity instead of at the origin, Theorem~\ref{propo:selfsimilar}, and the forward smoothing effect, Theorem~\ref{thm:Forward Mp smoothing effect}.

\noindent\textup{(iii)} We follow the construction in \cite{chasvaz} for the local case. Let $\varphi_1$ be a smooth, nonnegative, compactly supported function, such that $\|\varphi_1\|_{1}=\|\varphi_1\|_\infty=1$. The rescaled sequence
\[
    \varphi_n(x)=\varphi_1(n^bx),\qquad b=\frac{2}N,
\]
satisfies $\|\varphi_n\|_{1} = n^{-2}$, $\|\varphi_n\|_\infty=1$. Let \( u_n \) be the solution to problem~\eqref{eq:main} with initial datum $u_0=\varphi_n$. By comparison, $ \|u_n(t)\|_{\infty}\leq  1$. Let $t_0>0$ fixed. For each  $n\in\mathbb{N}$ we define
$$
    k_n=\left(\frac{n}{\|u_n(t_0)\|_\infty}\right)^\frac1N.
$$
By the scaling properties of the fractional Laplacian,  the scaled functions
\begin{equation*}\label{eq:star}
    \eta_n (x,t)= {k_n^{N}}u_n(k_nx,t)
\end{equation*}
are solutions to the equation in~\eqref{eq:main}. They satisfy $\|\eta_n(0)\|_{1}=\|\varphi_n\|_{1}=n^{-2}$, $\| \eta_n(t_0) \|_{\infty} = n$.
Let $u$ be the solution corresponding to~\eqref{eq:main} with initial value
$$
     u(x,0) = \sum_{n=1}^{\infty} \eta_n(x,0).
$$
Clearly this function belongs to $L^{1}$:
$$
   \| u(0)\|_{1}\le \sum_{n=1}^{\infty}\|\eta_n(0)\|_{1}=\sum_{n=1}^{\infty}\|\varphi_n\|_{1}= \sum_{n=1}^{\infty}\frac 1{n^2}<\infty.
$$
Since $u(0)\ge \eta_n(0)$ for every $n\ge1$,  by~Theorem~\ref{thm:comparison} we have
$u(t)\geq \eta_n(t)$
for every $n$, and every $t>0$. In particular,
$$
\|u(t_0)\|_{\infty} \ge \|\eta_n(t_0)\|_{\infty}=n.
$$
Letting $n\rightarrow\infty$, we obtain $\|u(t_0) \|_{\infty} =\infty$, and $t_0>0$ is arbitrary. In other words, there is no $L^{1} $--$L^\infty$ smoothing effect. This gives the desired counterexample by using 
Theorem~\ref{thm:Forward Mp smoothing effect}.
\end{proof}

\begin{remark}
    If $u$ is a weak energy solution, then $u^m(t)\in \dot{H}^{\frac\sigma2}(\mathbb{R}^N)$ for almost every $t>0$. Therefore, by Sobolev's embedding, $u(t)\in L^{\frac{2mN}{N-\sigma}}(\mathbb{R}^N)$. On the other hand, if $m\in(\frac{N-\sigma}{N+\sigma},m_c)$, then $\frac{2Nm}{N-\sigma}>p^*$ and hence $u(t)$ would be bounded, because of~\cite[Theorem 8.2]{b1}. Thus, our counterexample is a very weak solution, but not a weak energy solution, at least in this range.
\end{remark} 

\begin{remark} For the critical exponent $p=p^*>1$, the counterexample showing the lack of any smoothing starting from $\mathcal{M}^{p^*}$ is the explicit solution~\eqref{sol-extinction}. We do not know if there are smoothing effects (forward or backward) starting from  an initial datum in $L^{p^*}$. An argument similar to the one used in Theorem~\ref{failure of $L^p$--$L^q$}(iii), with $m<m_c$, does not work as a consequence of the property of finite-time extinction, which implies that the solutions for the small masses $u_n$ that we use in the  construction will vanish in times that go to zero with $n$, and we can not define the coefficients~$k_n$. This is not a coincidence: the same obstruction affects our other main tool, concentration comparison (Proposition~\ref{propo:concentrationcomparison}). Indeed, the worst-case solution in $\mathcal{M}^{p^*}$ against which every comparison in this range is made is precisely the explicit, extinguishing solution~\eqref{sol-extinction}, so the resulting bound necessarily degenerates as $t\to T^-$ and cannot distinguish a genuine smoothing effect from mere extinction. Settling the question for $L^{p^*}$ therefore seems to require an approach different from the ones developed here.
\end{remark}

\appendix

\section{Continuity in \texorpdfstring{$p^*$}{p*}}

We  prove the continuity in time in $L^p(\mathbb{R}^N)$, away from $t=0$, of $L^p$ solutions for exponents satisfying $p>p^*$. 

\begin{theorem}[Continuity in $L^p$]\label{Lp continuity}
    Let $u_0 \in L^p(\mathbb{R}^N) $ with $p > p^*$. Then the corresponding $L^p$ solution \( u \) to~\eqref{eq:main} satisfies $u \in C((0,\infty); L^p(\mathbb{R}^N))$.
\end{theorem}

\begin{proof}
We will prove that for every $\varepsilon>0$ there exists $h>0$ small such that
$$
    \int_{\mathbb{R}^N}|u(x,t_1)-u(x,t_2)|^p\,dx< c(p)\varepsilon\qquad \textup{if }|t_1-t_2|<h,\ t_1,t_2>0,
$$
so in fact we have uniform continuity. We split the integral to be estimated as
$$
    \int_{\mathbb{R}^N}|u(x,t_1)-u(x,t_2)|^p\,dx=\underbrace{\int_{B_R}|u(x,t_1)-u(x,t_2)|^p\,dx}_{I_1}+\underbrace{\int_{B^c_R}|u(x,t_1)-u(x,t_2)|^p\,dx}_{I_2}
$$
for some $R>0$ to be chosen later. Take $0<\tau<\min\{t_1,t_2\}$ a time such that $u(\tau)\in L^p(\mathbb{R}^N)$. Thus there exists some large value $R_1=R_1(\varepsilon,\tau)$ such that 
$$
    \int_{B_R^c}u^p(x,\tau)\,dx<\varepsilon\quad\textup{for every }R>R_1.
$$
We claim that  there also exists $R_2=R_2(\varepsilon,\tau)>R_1$ such that
$$
    \int_{B^c_R}u^p(x,s)\,dx<2\varepsilon\quad\textup{for every }R>R_2\textup{ and almost every }s>\tau. 
$$
This would show that $I_2\le C_p \varepsilon$ for almost every $s>\tau$. 

To prove the claim, we estimate $\displaystyle\int_{\mathbb{R}^N}\phi_Ru^p(s)$, where $\phi_R(x)=\phi(x/R)$ and $\phi$ is a smooth, nonnegative, radially nonincreasing function such that
\[
    \phi(x)=
    \begin{cases}
        1, & |x|>1,\\
        0, & |x|<1/2.
    \end{cases}
\]
Multiplying the equation by $p\phi_R u^{p-1}(s)$, and using the formula for the fractional Laplacian of a product, we get 
$$    \partial_t\int_{\mathbb{R}^N}\phi_R(x)u^p(x,s)\,dx
    =-p\int_{\mathbb{R}^N} u^m(x,s)(-\Delta)^{\frac\sigma2}(u^{p-1}(x,s)\phi_R(x))\,dx=J^1+J^2,
$$
where
\begin{align*}
    J^1&
    =-p\left(\int_{\mathbb{R}^N} u^{m+p-1}(x,s)(-\Delta)^{\frac\sigma2}\phi_R(x)\,dx+\int_{\mathbb{R}^N}\phi_R(x)u^m(x,s)(-\Delta)^{\frac\sigma2}u^{p-1}(x,s)\,dx\right),
    \\
    J^2&=c\int_{\mathbb{R}^N}u^m(x,s)\mathcal{B}(u^{p-1},\phi_R)(x,s)\,dx,\quad
    \mathcal{B}(f,g)(x)=\int_{\mathbb{R}^N}\frac{(f(x)-f(y))(g(x)-g(y))}{|x-y|^{N+\sigma}}\,dy.
\end{align*}
Using Kato’s inequality,
$$
    (-\Delta)^{\frac\sigma2}u^{m+p-1}\le \frac{p+m-1}{p-1} u^m(-\Delta)^{\frac\sigma2}u^{p-1},
$$
we can combine the two integrals defining $J^1$ into just one, obtaining
\begin{align*}
    J^1&\le\int_{\mathbb{R}^N} u^{m+p-1}(x,s)|(-\Delta)^{\frac\sigma2}\phi_R(x)|\,dx 
    \le c\|u(s)\|_p^{m+p-1}\|(-\Delta)^{\frac\sigma2}\phi_R\|_{\frac p{1-m}}
    \\&
    \le c\|u_0\|_p^{m+p-1}R^{-\sigma+\frac{N(1-m)}{p}}\|(-\Delta)^{\frac\sigma2}\phi\|_{\frac p{1-m}}\le cR^{-\sigma+\frac{N(1-m)}{p}}.
\end{align*}
On the other hand,  we simplify the integral involving $\mathcal{B}$ using the estimate
$$
    u^m(x,s)u^{p-1}(y,s)\le c \left(u^{m+p-1}(x,s)+u^{m+p-1}(y,s)\right)
$$
which follows from Young’s inequality $a^p b^q \le \frac{p}{p+q}\, a^{p+q} + \frac{q}{p+q}\, b^{p+q}$, and symmetry. Thus,
\begin{align*}
J^2&\le c\int_{\mathbb{R}^N}u^m(x,s)\mathcal{B}(u^{p-1},\phi_R)(x,s)\,dx \\ &\le c
\int_{\mathbb{R}^N}u^{m+p-1}(x,s)\left(\int_{\mathbb{R}^N}\frac{|\phi_R(x)-\phi_R(y)|}{|x-y|^{N+\sigma}}\,dy\right)\,dx
\\&
 \leq cR^{-\sigma}\|u(t)\|_p^{m+p-1}\left(\int_{\mathbb{R}^N}\left( \int_{\mathbb{R}^N}\frac{|\phi_R(x)-\phi_R(y)|}{|x-y|^{N+\sigma}}\,dy\right)^{\frac{1-m}{p}}\,dx\right)^{\frac p{1-m}}
 \\&
\le  cR^{-\sigma+\frac{N(1-m)}{p}}\left(\int_{\mathbb{R}^N}\left( \int_{\mathbb{R}^N}\frac{|\phi(z)-\phi(w)|}{|z-w|^{N+\sigma}}\,dw\right)^{\frac{1-m}{p}}\,dz\right)^{\frac p{1-m}}
 \\&
\le  cR^{-\sigma+\frac{N(1-m)}{p}}.
\end{align*}
We therefore get,
$$
\partial_t\int_{\mathbb{R}^N}\phi_R(x)u^p(x,s)\,dx\le cR^{-\sigma+\frac{N(1-m)}{p}},
$$
which integrated over the time interval $(\tau,s)$  gives
$$
\int_{\mathbb{R}^N}\phi_R(x)u^p(x,s)\,dx\le \int_{\mathbb{R}^N}\phi_R(x)u^p(x,\tau)\,dx+cR^{-\sigma+\frac{N(1-m)}{p}}. 
$$
Since $p>p^*$, the exponent $-\sigma+\frac{N(1-m)}{p}$ is negative, and the claim follows if $R$ is big enough. 

Now we turn to $I_1$, which is easier to estimate. Since $p>p^*$, $u$ is bounded, and we get
$$
    I_1=\int_{B_R}|u(x,t_1)-u(x,t_2)|^p\,dx \le c(\tau,R)\int_{B_R}|u(x,t_1)-u(x,t_2)|\,dx<\varepsilon
$$
for any fixed $R$, if $|t_1-t_2|<h$ with $h$ small,  by the continuity in $L^1_{\text{loc}}$. 
\end{proof}

\begin{remark} Continuity in  $L^p$ with $1\le p\le p^*$ and in $\mathcal{M}^p$ with $p>1$ is open. If $p>p^*$, we could try to adapt the argument used to prove the continuity in $L^p$ in Theorem~\ref{Lp continuity} to deal with $\mathcal{M}^p$: one obtains essentially the same tail decay estimate as for $L^p$ solutions, up to an additional term depending on $R$. However,  one cannot choose $R$ in such a way that the tails become small. As a counterexample, consider again $U_0(x)=|x|^{-\frac{N}{p}}$. Continuity up to $t=0$ is also an open problem in spaces of both types.
\end{remark}

%%%%%%%%%%%%%%%%%%%%%%%%%%%%%%%%%%%%
\section*{Acknowledgements}
The three authors were supported by MICIU/AEI/10.13039/501100011033 (Spain), through  grants PID2023-146931NB-I00, RED2024-153842-T and CEX2023-001347-S.

We thank Jorge Ruiz-Cases for some fruitful discussions.
%%%%%%%%%%%%%%%%%%%%%%%%%%%%%%%%%%%%%%%%%%

\end{document}